\documentclass[11pt]{amsart}
\usepackage{stmaryrd}
\usepackage{amssymb}
\usepackage{csvsimple}
\usepackage{siunitx}
\usepackage{hyperref}
\usepackage{url}
\usepackage{graphicx,color}
\usepackage{lineno}           
\usepackage[style=trad-abbrv,url=false,doi=true,
  eprint=true,isbn=false,backend=biber]{biblatex}

\def\R{\mathbb{R}}

\def\cA{\mathcal{A}}

\def\cI{\mathcal{I}}
\def\cJ{\mathcal{J}}

\def\cM{\mathcal{M}}
\def\cN{\mathcal{N}}

\def\cS{\mathcal{S}}
\def\cT{\mathcal{T}}

\def\a{\alpha}
\def\b{\beta}
\def\g{\gamma}
\def\d{\delta}

\def\k{\kappa}
\def\l{\lambda}

\def\p{\partial}
\def\o{\omega}
\def\veps{\varepsilon}
\def\vrho{\varrho}
\def\vphi{\varphi}
\def\O{\Omega}
\def\G{\Gamma}

\def\transp{{\sf T}}

\def\tu{\widetilde{u}}

\def\tx{\widetilde{x}}

\newcommand{\dv}[1]{\,{\mathrm d}#1}
\newcommand{\dual}[3][]{#1\langle #2,#3#1\rangle}

\newcommand{\wcheck}[1]{#1\hspace{-.8ex}\mbox{\huge {\lower.45ex \hbox{$\textstyle \check{}$}}} \hspace{.5ex}}

\DeclareMathOperator{\supp}{supp}

\let\oldmarginpar\marginpar
\renewcommand\marginpar[1]{
  \oldmarginpar[\raggedleft\footnotesize #1]
  {\raggedright\footnotesize #1}}

\newtheorem{definition}{Definition}
\newtheorem{lemma}[definition]{Lemma}

\newtheorem{theorem}[definition]{Theorem}

\newtheorem{example}[definition]{Example}

\numberwithin{definition}{section}
\definecolor{modmag}{RGB}{179,0,229}

\renewcommand{\text}{\textnormal}
\def\inv{{\rm inv}}
\def\DD{{\rm D}}
\def\NN{{\rm N}}
\def\tl{\widetilde{\l}}
\def\tPi{\widetilde{\Pi}}
\def\tF{\widetilde{F}}

\def\tv{\widetilde{v}}
\def\eoc{{\rm eoc}}

\begin{document}
\title[Optimal approximation of harmonic maps]{Quasi-optimal error estimates
for the approximation of stable harmonic maps}
\author[S. Bartels]{S\"oren Bartels}
\address{Abteilung f\"ur Angewandte Mathematik,
Albert-Ludwigs-Universit\"at Freiburg, Hermann-Herder-Str.~10,
79104 Freiburg i.~Br., Germany}
\email{bartels@mathematik.uni-freiburg.de}
\author[C. Palus]{Christian Palus}
\address{Abteilung f\"ur Angewandte Mathematik,
Albert-Ludwigs-Universit\"at Freiburg, Hermann-Herder-Str.~10,
79104 Freiburg i.~Br., Germany}
\email{christian.palus@mathematik.uni-freiburg.de}
\author[Z. Wang]{Zhangxian Wang}
\address{Abteilung f\"ur Angewandte Mathematik,
Albert-Ludwigs-Universit\"at Freiburg, Hermann-Herder-Str.~10,
79104 Freiburg i.~Br., Germany}
\email{zhangxian.wang@mathematik.uni-freiburg.de}
\date{\today}
\renewcommand{\subjclassname}{
\textup{2010} Mathematics Subject Classification}
\subjclass[2010]{35J62 (35J50 35J57 65N30)}
\begin{abstract}
Based on a quantitative version of the inverse function theorem
and an appropriate saddle-point formulation we derive a
quasi-optimal error estimate for the finite element approximation
of harmonic maps into spheres with a nodal discretization of
the unit-length constraint. The estimate holds under
natural regularity requirements and appropriate geometric
stability conditions on solutions. Extensions to other target manifolds
including boundaries of ellipsoids are discussed.
\end{abstract}
\keywords{Harmonic maps, finite elements, inverse function theorem,
saddle-point formulation, error estimate}

\maketitle

\section{Introduction}
Harmonic maps into spheres are stationary configurations for the
Dirichlet energy  
\[
I(u) = \frac12 \int_\O |\nabla u|^2 \dv{x}
\]
among vector fields $u:\O\to \R^m$, $\O\subset \R^d$,
satisfying prescribed boundary
conditions $u|_{\G_\DD} = u_\DD$ and the pointwise sphere constraint
\[
u(x) \in S^{m-1} \quad \Longleftrightarrow \quad |u(x)|^2 -1 = 0
\]
for almost every $x\in \O$. The existence of global minimizers
is an immediate consequence of the direct method in the calculus
of variations provided that the admissible set is non-empty. More
generally, stationary points satisfy the Euler--Lagrange equations
\begin{equation}\label{eq:euler_lagrange}
-\Delta u = |\nabla u|^2 u, \quad u|_{\G_\DD} = u_\DD, \quad
\p_n u|_{\G_\NN} = 0, \quad |u|^2 = 1,
\end{equation}
where $\G_\NN = \p\O \setminus \G_\DD$. Since the right-hand side
in the partial differential equation may only belongs
to $L^1(\O;\R^m)$ regularity of solutions cannot be expected
in general and in fact solutions that are everywhere discontinuous
exist, cf.~\cite{Stru96-book,Rivi95}.

Motivated by related models and applications in micromagnetics,
liquid crystal devices, and nonlinear bending, cf.,
e.g.,~\cite{KPPRS19,BoNoWa20,BaBoNo17} and references therein, the numerical
approximation of pointwise constrained variational problems has received
considerable attention
in the last decades. Various discretizations and iterative schemes have been devised
and analyzed in~\cite{LinLus89,Alou97,Bart05,Bart10,GrHaSa15}. To avoid unjustified
regularity assumptions, the convergence of numerical methods
has often been based on weak compactness results for the
Euler--Lagrange equations which shows that weak accumulation
points of approximations are harmonic maps. To fully justify
the methods it is important to prove their optimal convergence
in the case of sufficiently regular solutions, and only a few results
in this direction are available, cf.~\cite{ClaDzi03,QiTaWi09,GrHaSa15}.

An attractive and flexible approach to deriving error estimates for numerical
schemes has been identified in~\cite{QiTaWi09} and it is our aim to
address its validity for three-dimensional domains~$\O$ and higher-dimensional
target manifolds. Their approach is based on the
Lagrange functional
\[
L(u,\l) = \frac12 \int_\O |\nabla u|^2 \dv{x}
+ \frac12 \int_\O \l (|u|^2 -1) \dv{x}
\]
that imposes the constraint via a Lagrange multiplier~$\l$.
A suitable functional analytical framework interprets the
constraint term in a weaker sense and seeks stationary
pairs $(u,\l)$ in the affine space
\[
\cA = (u_\DD,0) + X,
\]
with the product space
\[
X = H^1_\DD(\O;\R^m) \cap L^\infty(\O;\R^m) \times H^{-1}(\O),
\]
where $H^{-1}(\O)$ is the topological dual of the
Sobolev space $H^1_\DD(\O)$. To derive error estimates
in a neighborhood of a solution $(u,\l)$  the mapping properties
of the second variation of $L$ are releavant. Its stable
invertibility can be analyzed in terms of a saddle-point
problem which seeks for a given functional $(f,g)\in X'$ a solution
$(v,\mu) \in X$ such that
\[\begin{split}
(\nabla v,\nabla w) + \dual{\l}{v\cdot w} + \dual{\mu}{u \cdot w} &= \dual{f}{w}, \\
\dual{\eta}{u\cdot v} \, \hphantom{+ \dual{\l}{v\cdot w} + \dual{\mu}{u \cdot w}}  & = \dual{g}{\eta},
\end{split}\]
for all $(w,\eta) \in X$. Well established theories for
saddle-point problems assert that the problem has a unique
and stable solution if and only if the bilinar form
\[
b_u(\mu,v) = \dual{\mu}{u \cdot v}
\]
is bounded and satisfies an inf-sup condition, and the bilinear form
\[
a_\l(v,w) = (\nabla v,\nabla w) + (\l, v\cdot w),
\]
with $\l= -|\nabla u|^2$, is bounded and defines an invertible operator on the
kernel of $b_u$ with respect to the second argument.
The inf-sup condition is obtained by choosing
for given $\mu\in H^{-1}(\O)$ the function $v= \phi u$, where
$\phi \in H^1_\DD(\O)$ satisfies $\dual{\mu}{\phi} = \|\mu\|_{H^{-1}}$.
The kernel of $b_u$ consists of
tangential vector fields $v\in T_u$ with
\[
T_u = \big\{ v\in H^1_\DD(\O;\R^m) \cap L^\infty(\O;\R^m):
v\cdot u = 0 \text{ a.e.}\big\}.
\]
We say that $u$ is a stable harmonic map, if $a_\l$ is $H^1$ coercive
on $T_u$. Besides the special case
$|\nabla u| < c_P^{-1}$ with the Poincar\'e constant $c_P>0$
a coercivity result holds if the one-dimensional sphere
is considered as a target manifold, i.e., $m=2$ and $u:\O \to S^1$.
In this case tangential vector fields are given by
\[
v = \a u^\perp,
\]
with $\a\in H^1_\DD(\O)$ and the rotation $u^\perp$ of $u$ by $\pi/2$.
We then have the coercivity property
\[
a_\l (v,v) = \int_\O |\nabla \a|^2 \dv{x} \ge (1+\|\nabla u\|_{L^\infty}^2 c_P^2)^{-1} \|\nabla v\|^2,
\]
whenever the harmonic map~$u$ satisfies $u\in W^{1,\infty}(\O;\R^2)$.
Remarkably, this stability property fails if the (same) harmonic
map $u$ is allowed to attain values in the two-dimensional sphere.
Indeed, by embedding the image of $u$ into $S^2$ via
$\tu = [u,0]^\transp$, and considering $v = \a e_3 \in T_{\tu}$ we find that
\[
a_\l(v,v) = \int_\O |\nabla \a|^2 - |\nabla u|^2 \a^2 \dv{x}.
\]
The right-hand side can only be positive for all $\a\in H^1_\DD(\O)$
if $|\nabla u|$ is sufficiently small.

Only a few results are available concerning the uniqueness and stability of harmonic
maps into higher-dimensional spheres, cf.,~e.g.,~\cite{JagKau79,JagKau83}.
In particular, if a cut-locus condition is satisfied, e.g., if the image of a harmonic
map is strictly contained in a hemisphere, then~\cite[Theorem B]{JagKau79}
states that the only Jacobi field along a harmonic map $u$, i.e., a field
$v\in T_u$ with $a_\l(v,v)=0$, is the trivial one. If $u\in \cA$ is an
absolute minimizer for $I$ then we have that $a_\l$ is semi-definite
and if, e.g., $u\in W^{1,\infty}(\O;\R^m)$ a contradiction argument implies that
$a_\l$ is coercive on $T_u$. In view of limited
regularity properties, cf.~\cite{SchUhl82,Lin87,Rivi95} and nonuniqueness
properties, cf., e.g.,~\cite{Bart15-book}, a more general theory cannot be
expected.

Provided that the harmonic map $u$ is regular, i.e., we
have that $u\in H^2(\O;\R^m)\cap W^{1,\infty}(\O;\R^m)$, and stable,
i.e., the bilinear form $a_\l$ is $H^1$ coercive on $T_u$, we derive
the quasi-optimal error estimate
\[
\|\nabla (u-u_h) \| + \|\l - \l_h \|_{H^{-1}} \le c_u h,
\]
for a canonical discretization of the Lagrange
functional and the unique finite element
solution~$(u_h,\l_h) \in \cS^1(\cT_h)^m \times \cS^1_\DD(\cT_h)$
in an appropriate neighborhood of~$u$. Our analysis thus shows
that the arguments of~\cite{QiTaWi09} also apply to
higher-dimensional domains and targets under appropriate and
meaningful conditions. Some restrictions arise from the simpler
functional analytical framework in the discrete setting and the
resulting use of inverse estimates to control $L^\infty$ norms.

The outline of the article is as follows. Some preliminaries are
stated in Section~\ref{sec:prelim}. The main error estimate
is derived in Section~\ref{sec:main_est} by verifying the conditions
of the inverse function theorem. The application of the analysis
to other target manifolds is addressed in Section~\ref{sec:targets}.
Numerical experiments that confirm the theoretical results are
reported in Section~\ref{sec:num_ex}.

\section{Preliminaries}\label{sec:prelim}
We use standard notation to denote Lebesgue and Sobolev spaces. The
integration domain is often omitted in norms and we abbreviate
the inner product and norm in $L^2(\O;\R^\ell)$ by $(\cdot,\cdot)$
and $\|\cdot\|$, respectively. Throughout the article $c>0$ denotes
a factor that may depend on regularity properties of a fixed solution~$u$
but not on the mesh-sizes of a sequence of triangulations; the dependence
on~$u$ is occasionally indicated via a subindex. We let
$c_P>0$ denote the smallest positive number
with $\|v\| \le c_P \|\nabla v\|$ for all $v\in H^1_\DD(\O)$; we remark
that $c_P\le d_\O/\pi$ if $\G_\DD = \p\O$ and $\O$ is a convex
domain with diameter $d_\O$, cf.~\cite{PayWei60}.

\subsection{Finite element functions}
For a regular and quasi-uniform triangulation $\cT_h$ of the simplicial domain
$\O \subset \R^d$  with mesh-size $h>0$ we denote the $C^0$ conforming
finite element space by $\cS^1(\cT_h)$ of elementwise linear functions.
We denote the subspace of functions vanishing on $\G_\DD$ by
\[
\cS^1_\DD(\cT_h) = \cS^1(\cT_h) \cap H^1_\DD(\O).
\]
We let $\cN_h$ be the set of vertices of elements and denote the nodal interpolation
operator applied to scalar or vector-valued functions by
\[
\cI_h : C(\overline{\O};\R^\ell) \to \cS^1(\cT_h)^\ell, \quad \cI_h v = \sum_{z\in \cN_h} v(z) \vphi_z,
\]
where $(\vphi_z:z\in \cN_h)$ is the scalar nodal basis for $\cS^1(\cT_h)$.
We note that we have the nodal interpolation estimate for $v\in H^1_\DD(\O;\R^\ell)$
with $v|_T\in H^2(T)$ for all $T\in \cT_h$ that
\[
\|v-\cI_h v\| + h\|\nabla (v-\cI_h v)\|  \le c h^2 \|D_h^2 v\|,
\]
where $D_h^2$ denotes the elementwise application of the Hessian.
For an elementwise poynomial function $\phi_h \in H^1(\O)$ we have
\[
\|\phi_h - \cI_h \phi_h \|_{L^1} \le c h^2 \|D_h^2 \phi_h\|_{L^1}.
\]
We make repeated use of inverse estimates, which read for $v_h \in \cS^1_\DD(\cT_h)$
\begin{equation}\label{eq:inv_est_der}
\|\nabla v_h\|_{L^p} \le c h^{-1} \|v_h\|_{L^p}
\end{equation}
and, using Sobelev inequalities, with
$\g_\inv(h) = 1, 1+ |\log h|, h^{-1/2}$ for $d=1,2,3$, respectively,
we moreover have that
\begin{equation}\label{eq:inv_est_inf}
\|v_h\|_{L^\infty} \le  c \g_\inv(h) \|\nabla v_h\|.
\end{equation}
The estimate can be deduced from elementary local norm equivalences and Sobolev
inequalities, i.e.,
\[
\|v_h\|_{L^\infty} \le c h^{-d/p} \|v_h\|_{L^p} \le c h^{-d/p}  \|\nabla v_h\|
\]
with $p\le \infty$, $p<\infty$, and $p\le 2d$, for $d=1,2,3$,
respectivly. A precise characterization of
the Sobolev embedding is needed if $d=2$, cf.~\cite{Bart15-book}, a weaker
result for $d=2$ is obtained with $p=d/\veps$ for fixed $\veps>0$.
A discrete inner product is for $v,w\in C(\overline{\O})$ defined via
\[
(v,w)_h = \int_\O \cI_h (v\cdot w)\dv{x} = \sum_{z\in \cN_h} \b_z v(z) \cdot w(z),
\]
where $\b_z = \int_\O \vphi_z \dv{x}$ is positive. For $v_h \in \cS^1(\cT_h)$
we have $\|v_h\|_h \le \|v_h\| \le c \|v_h\|_h$. We frequently use the following
estimate.

\begin{lemma}[Quadrature control]\label{la:quad_control}
For $\psi_h\in \cS^1_\DD(\cT_h)$ and $\phi \in C(\overline{\O})$ with
$\phi|_T \in H^2(T)$ for all $T\in \cT_h$ we have
\[
\big|(\psi_h, \phi)_h - (\psi_h,\phi)\big|
\le c h^2 \big(\|\nabla \psi_h\| \|\nabla \cI_h \phi\| + \|\psi_h\| \|D_h^2 \phi\|\big).
\]
In case of an elementwise polynomial function $\phi_h\in C(\overline{\O})$ we
have
\[
\big|(\psi_h, \phi_h)_h - (\psi_h,\phi_h)\big| \le c h \|\psi_h\| \|\nabla \phi_h\|.
\]
\end{lemma}

\begin{proof}
We have that
\[ 
(\psi_h,\phi)_h - (\psi_h,\phi)
= \int_\O \cI_h (\psi_h \phi) - \psi_h \cI_h \phi \dv{x}
+ \int_\O \psi_h (\cI_h \phi - \phi) \dv{x},
\]
and the two terms on the right-hand side are controlled with the $L^1$ and
$L^2$ nodal interpolation estimates stated above. The second estimate follows
from the first one by using the inverse estimate~\eqref{eq:inv_est_der}
(generalized to elementwise polynomial functions) twice
and the $H^1$ stability of $\cI_h$ on elementwise polynomial functions.
\end{proof}

We let $\Pi_h : L^2(\O)\to \cS^1_\DD(\cT_h)$ denote the $L^2$ projection onto
$\cS^1_\DD(\cT_h)$ and by $\tPi_h: L^2(\O) \to \cS^1_\DD(\cT_h)$ the modified version given by
\[
(\tPi_h v, \phi_h)_h = (v,\phi_h)
\]
for all $\phi_h\in \cS^1_\DD(\cT_h)$. We note that $\Pi_h$ is $H^1$ stable on quasi-uniform
triangulations. The modified projection has similar properties as $\Pi_h$.

\begin{lemma}[Modified $L^2$ projection]\label{la:l2proj}
The projection $\tPi_h$ satisfies for all $v \in H^1_\DD(\O)$
\[
\| \nabla \tPi v\| + h^{-1} \| \tPi v -v\| \leq c \|\nabla v\|.
\]
\end{lemma}

\begin{proof}
With the standard $L^2$ projection $\Pi_h$ onto $V_h$,
define $\delta_h = \tPi v - \Pi_h v$. We then have
\[
\|\delta_h\|^2 \le \|\delta_h\|_h^2 = (\delta_h, \tPi v - \Pi_h v)_h
= (\delta_h,\Pi_h v) - (\delta_h,\Pi_h v)_h .
\]
Therefore, using Lemma~\ref{la:quad_control}, estimate~\eqref{eq:inv_est_der},
and the $H^1$-stability of $\Pi_h$ we find that
\[
\|\delta_h\|^2 \le ch^2 \|D_h^2(\delta_h \cdot \Pi_h v)\|_{L^1}
= ch^2 \|\nabla \delta_h\| \|\nabla \Pi_h v\|\le ch \|\delta_h\| \|\nabla v\|.
\]
Hence $\|\delta_h\| \le ch \|\nabla v\|$ and another application of an
inverse estimate yields $\|\nabla \delta_h\| \le c\|\nabla v\|$.  We therefore get
\[
\|\nabla \tPi v\| \le \| \nabla \Pi_h v\| + \|\nabla \delta_h \| \le c\|\nabla v\|.
\]
The error estimate follows from a related estimate for $\Pi_h$.
\end{proof}

We often use the dual space $H^{-1}(\O) = (H^1_\DD(\O))'$ which is equipped
with the operator norm
\[
\|\mu\|_{H^{-1}} = \sup_{\phi \in H^1_\DD(\O) \setminus \{0\}} \frac{\dual{\mu}{\phi}}{\|\nabla \phi\|}.
\]
We have the inverse estimate
\[
\|\mu_h \| \le c h^{-1} \|\mu_h\|_{H^{-1}}
\]
for all $\mu_h\in \cS^1_\DD(\cT_h)$.
The Cl\'ement quasi-interpolation operator $\cJ_h:L^1(\O) \to \cS^1(\cT_h)$
is with the sets $\o_z = \supp \vphi_z$, $z\in \cN_h$, defined via
\[
\cJ_h \a = \sum_{z\in \cN_h} \a_z \vphi_z, \quad \a_z = |\o_z|^{-1}  \int_{\o_z} \a \dv{x}.
\]
The variant $\cJ_{h,\DD}: L^1(\O) \to \cS^1_\DD(\cT_h)$ is obtained by setting
$\a_z= 0$ for all $z\in \cN_h\cap \G_\DD$. We remark that we have
\[
(\cJ_{h,\DD} \a, v)_h = (\cJ_h \a,v)_h
\]
for $v\in C(\overline{\O})$ with $v|_{\G_\DD} = 0$.
For $\a\in H^1(\O)$ we have
\[
\|\a-\cJ_h \a\| \le c h \|\nabla \a\|.
\]
A similar estimate holds for $\a\in H^1_\DD(\O)$ and $\cJ_{h,\DD}\a$,
cf., e.g.,~\cite{Bart15-book}.

\subsection{Inverse function theorem}
As in~\cite{DziHut06,QiTaWi09}
we use the following quantitative version of the inverse function theorem
to derive a local error estimate.

\begin{theorem}[Inverse function theorem]\label{quant_inv_fn}
Suppose that $F:X\to X'$ is continuous and assume that
$\tx \in X$ satisfies $\|F(\tx)\|_{X'}\le \k$.  If there
exist $c_L',c_\inv,\veps >0$ such that $F$ is Fr\'echet differentiable
in $B_\veps(\tx)$, with $DF(\tx)$ invertible, and
\[\begin{split}
\|DF(\tx)^{-1}\|_{L(X',X)} & \le c_\inv, \\
\|DF(x_1)-DF(x_2)\|_{L(X,X')} & \le c_L' \|x_1-x_2\|_X
\end{split}\]
for all $x_1,x_2\in B_\veps(\tx)$ with $\veps >0$ so that
$c_L'c_\inv \veps \le 1/2$ and $\k \le \veps/(2 c_\inv)$,
then there exists a unique $x\in B_\veps(\tx)$ such that $F(x)=0$.
\end{theorem}

\begin{proof}
The result is an immediate conseqence of the proof of~\cite[Thm.~3.1.5, p.~113]{Berg77-book}.
\end{proof}

We remark that if $F$ is defined on an affine space $\cA = x_\DD + X$ then the theorem
can be applied to $\tF(x) = F(x_\DD+x)$. The theorem also implies the superlinear
convergence of the Newton-type iteration $x^{k+1} = x^k - DF(\tx)^{-1} F(x^k)$ and
of the classical Newton iteration if a bound on the the inverse of the Jacobian
holds in $B_\veps(\tx)$. For quadratic convergence, a bound on the second variation
of $F$ is required.

\section{Error estimate}\label{sec:main_est}
We recall that harmonic maps into spheres are defined as
stationary pairs $(u,\l)\in \cA$ for the functional
\[
L(u,\l) = \frac12 \int_\O |\nabla u|^2 \dv{x} + \frac12 \dual{\l}{|u|^2 -1}
\]
An optimal pair satisfies the Euler--Lagrange equations~\eqref{eq:euler_lagrange}
with
\[
\l = -|\nabla u|^2.
\]
A finite element approximation is sought in the space of admissible pairs
\[
\cA_h = (u_{\DD,h},0) + X_h,
\]
with $u_{\DD,h} = \cI_h \tu_\DD$ for a continuous extension~$\tu_\DD$ of $u_\DD$
and the homogeneous space $X_h$ defined via
\[
X_h = \cS^1_\DD(\cT_h)^m \times \cS^1_\DD(\cT_h) \subset H^1_\DD(\O;\R^m) \times H^{-1}(\O).
\]
Here, no uniform bounds are included in the definition of $X_h$
in order to have a Hilbert space structure. Discrete harmonic maps
are stationary configurations for the functional
\[
L_h(u_h,\l_h) = \frac12 \int_\O |\nabla u_h|^2 \dv{x}
+ \frac12 \int_\O \cI_h \big[\l_h (|u_h|^2 -1) \big]\dv{x}.
\]
Discrete harmonic maps $(u_h,\l_h) \in \cA_h$ satisfy, cf.~\cite{QiTaWi09,Bart15-book},
\[\begin{split}
(\nabla u_h,\nabla v_h) + (\l_h,u_h\cdot v_h)_h &= 0, \\
(\mu_h, |u_h|^2 - 1)_h  \hphantom{+ (\l_h,u_h\cdot v_h)_h}  & = 0,
\end{split}\]
for all $(v_h,\mu_h) \in X_h$.
The saddle-point system can be formulated as a nonlinear equation with
a mapping $F_h: \cA_h \to X_h'$ via
\[
F_h(u_h,\l_h)[(v_h,\mu_h)] = (\nabla u_h,\nabla v_h) + (\l_h,u_h\cdot v_h)_h + (\mu_h, |u_h|^2 - 1)_h.
\]
The variational derivative of $F_h$ is given by
\[\begin{split}
DF_h(u_h,\l_h)[(v_h,\mu_h),& (w_h,\eta_h)] = (\nabla v_h,\nabla w_h) \\
& + (\l_h,w_h\cdot v_h)_h + (\mu_h, u_h\cdot w_h)_h + (\eta_h,u_h \cdot v_h).
\end{split}\]
To investigate the invertibility of the linear operator
$DF_h(\tu_h,\tl_h): X_h\to X_h'$ we resort to established
theories for linear saddle-point problems on Hilbert spaces and define for a given
pair $(\tu_h,\tl_h)$ the  bilinear forms
\begin{equation}\label{eq:bilinear_forms_discrete}
\begin{split}
a_{\tl_h}(v_h,w_h) &= (\nabla v_h,\nabla w_h) + (\tl_h,w_h\cdot v_h)_h, \\
b_{\tu_h}(\mu_h,v_h) &= (\mu_h, \tu_h\cdot v_h)_h,
\end{split}
\end{equation}
for all $v_h,w_h\in \cS^1_\DD(\cT_h)^m$ and $\mu_h \in \cS^1_\DD(\cT_h)$.
The invertibility is equivalent to the existence of a unique solution
$(v_h,\mu_h) \in X_h$ for every right-hand side $(f_h,g_h) \in X_h'$
such that
\[\begin{split}
a_{\tl_h}(v_h,w_h) + b_{\tu_h}(\mu_h,w_h) &= (f_h,w_h), \\
b_{\tu_h}(\eta_h, v_h) \, \hphantom{+ b_{\tu_h}(\mu_h,w_h)}  &= (g_h,\eta_h),
\end{split}\]
for all $(w_h,\eta_h)\in X_h$. Sufficient for this is that $a_{\tl_h}$ is coercive
on the kernel of $b_{\tu_h}$ and that $b_{\tu_h}$ satisfies an inf-sup condition,
cf.~\cite{Babu70,Brez74}.

\begin{lemma}[Invertibility]\label{la:invert}
(i) For every $\tu_h\in \cS^1(\cT_h)^m$ the
bilinear form $b_{\tu_h}$ satisfies the inf-sup condition
\[
\sup_{v_h \in \cS^1_\DD(\cT_h)^m \setminus \{0\}} \frac{b_{\tu_h}(\mu_h, v_h)}{\|\nabla v_h\|}
\ge c \|\tu_h\|_{W^{1,\infty}}^{-1} \|\mu_h\|_{H^{-1}}
\]
for all $\mu_h \in \cS^1_\DD(\cT_h)$. Moreover $b_{\tu_h}$ is continuous
with bound $c\|\tu_h\|_{W^{1,\infty}}$. \\
(ii) Assume that the the pair $(u,\l) \in \cA$ satisfies
\begin{equation}\label{eq:reg_cond}
u \in H^2(\O;\R^m) \cap W^{1,\infty}(\O;\R^m), \quad \l \in  H^1(\O) \cap L^\infty(\O),
\end{equation}
and that there exists $c_a>0$ such that
\begin{equation}\label{eq:stable_hm}
a_\l(v,v) \ge c_a \|\nabla v\|^2 \quad \text{for all $v\in T_u$}.
\end{equation}
Define  $(\tu_h,\tl_h) \in \cA_h$ via
\[
\tu_h= \cI_hu, \quad \tl_h = \cJ_{h,\DD} \l.
\]
Then for~$h$ sufficiently small we have
\[
a_{\tl_h}(v_h,v_h) \ge (c_a/2) \|\nabla v_h\|^2
\]
for all $v_h \in \cS^1_\DD(\cT_h)^m)$ with $\cI_h (v_h\cdot \tu_h)=0$.
Moreover, $a_{\tl_h}$ is continuous with bound $c \|\tl_h\|$. \\
(iii) Under the conditions of~(ii) the operator $DF_h(\tu_h,\tl_h)$ is
invertible with $\|DF_h(\tu_h,\tl_h)\|_{L(X_h',X_h)} \le c_\inv$
for a constant~$c_\inv>0$ that depends on $\|u\|_{W^{1,\infty}}$,
$\|\l\|$, and $c_a$. The smallness condition on $h$ additionally
depends on $\|\nabla \l\|$ and $\|D^2u\|$.
\end{lemma}

\begin{proof}
(i) To verify the inf-sup condition for $b_{\tl_h}$ we follow~\cite{QiTaWi09}
and note that the Hahn--Banach theorem implies that for given $\mu_h\in \cS^1_\DD(\cT_h)$
there exists $\phi\in H^1_\DD(\O)$ with $\|\nabla \phi\| = 1$ and
\[
(\mu_h,\phi) = \|\mu_h\|_{H^{-1}}.
\]
With the modified $L^2$ projection $\tPi_h$ we define
\[
v_h = \cI_h ((\tPi_h \phi) \tu_h).
\]
Since $|\tu_h(z)|^2 = 1$ for all $z\in \cN_h$ this choice implies that we have
\[
b_{\tu_h}(\mu_h, v_h) = (\mu_h, (\tPi_h \phi) \tu_h \cdot \tu_h)_h
= (\mu_h, \tPi_h \phi)_h = (\mu_h,\phi) = \|\mu_h\|_{H^{-1}}.
\]
Using the $H^1$-stability of $\cI_h$ on elementwise polynomials and
the $H^1$ stability of $\tPi_h$ on quasi-uniform meshes, we find that
\[
\|\nabla v_h\| \le c \|\nabla( (\tPi_h \phi) \tu_h)\| \le c \|\nabla \tPi_h \phi\| \|\tu_h\|_{W^{1,\infty}}
 \le c \|\nabla \phi\| \|\tu_h\|_{W^{1,\infty}},
\]
i.e., $\|\nabla v_h\| \| \tu_h \|_{W^{1,\infty}}^{-1} \le c$.
Combining the last two estimates leads to
\[
b_{\tu_h}(\mu_h, v_h)  \ge c \|\tu_h\|_{W^{1,\infty}}^{-1} \|\nabla v_h\| \|\mu_h\|_{H^{-1}},
\]
which is the asserted inf-sup property. Using Lemma~\ref{la:quad_control}
and an inverse estimate, we verify the boundedness of $b_{\tu_h}$, i.e.,
\[\begin{split}
|b_{\tu_h}(\mu_h,v_h)| &\le |(\mu_h, \tu_h \cdot v_h)| + c h \|\mu_h\| \|\nabla (\tu_h \cdot v_h)\| \\
&\le \|\mu_h\|_{H^{-1}} \|\tu_h\|_{W^{1,\infty}} \|\nabla v_h\|,
\end{split}\]
(ii) Given $v_h \in \cS^1_\DD(\cT_h)^m$ with $\cI_h (v_h\cdot \tu_h)=0$
the function
\[
\tv^h = v_h - (v_h \cdot u) u
\]
satisfies $\tv^h \in T_u$ and hence we have
$a_\l(\tv^h,\tv^h) \ge c_a \|\nabla \tv^h\|^2$. Using that we may
replace $\tl_h$ by $\cJ_h \l$ in $a_{\tl_h}$, Lemma~\ref{la:quad_control}
and~\eqref{eq:inv_est_inf} lead to
\[
\big|(\tl_h, |v_h|^2)_h - (\cJ_h \l, |v_h|^2)\big|
\le c h \g_\inv(h) \|\cJ_h\l \| \|\nabla v_h\|^2.
\]
With this estimate we find that
\[\begin{split}
\big|a_{\tl_h}(v_h, & v_h) - a_\l(\tv^h,\tv^h) \big| \\
& \le \|\nabla (v_h- \tv^h)\| \big( \|\nabla v_h\| + \|\nabla \tv^h\|\big)
+ c h \g_\inv (h) \| \cJ_h \l \| \|\nabla v_h \|^2 \\
& \quad
+ \|\cJ_h \l\| \|v_h - \tv_h\|_{L^4} \big(\|v_h\|_{L^4} + \|\tv^h\|_{L^4}
+ \|\cJ_h \l - \l\| \| \tv^h\|_{L^4}^2\big).
\end{split}\]
To bound the terms on the right-hand side we note that
$\cI_h ((v_h\cdot u)u) = 0$ and hence
\[\begin{split}
\|\nabla (v_h -\tv^h)\| &\le c h \|D_h^2 ((v_h \cdot u)u) \| \\
& \le c h \big(\|\nabla v_h\| \|\nabla u\|_{L^\infty}
+ \|v_h\|_{L^\infty} (\|D^2 u\| + \|\nabla u\|_{L^\infty}^2) \big) \\
& \le c h \g_\inv(h) \|\nabla v_h\|.
\end{split}\]
A Sobolev embedding and a Poincar\'e inequality show that the same
bound applies to $\|v_h -\tv^h\|_{L^4}$. Moreover, we have that
\[
\|\tv^h \| + \|\nabla \tv^h\| + \|\tv^h\|_{L^4}  + \|v_h\|_{L^4} \le c \|\nabla v_h\|.
\]
Noting stability and approximation properties of the Cl\'ement
quasi-interpolant, the combination of the estimates implies that
\[
a_{\tl_h}(v_h,v_h) \ge  a_\l(\tv^h,\tv^h) - c h\g_\inv(h) \|\nabla v_h\|^2,
\]
which is the asserted coercivity property. Finally, as
a consequence of Lemma~\ref{la:quad_control} and inverse
estimates, $a_{\tl_h}$ satisfies the bound
\[\begin{split}
|a_{\tl_h}(v_h,w_h)| & \le \|\nabla v_h\| \|\nabla w_h\|
+ |(\tl_h,v_h \cdot w_h)| + c h \|\tl_h\| \|\nabla (v_h \cdot w_h)\| \\
&\le (1+ c h \g_\inv(h)) \|\tl_h\| \|\nabla v_h\| \|\nabla w_h\|.
\end{split}\]
(iii) The inf-sup condition for $b_{\tu_h}$ and the
coercivity of $a_{\tl_h}$ on the kernel of $b_{\tu_h}$, which is given by
$\ker b_{\tu_h} = \{ v_h \in \cS^1_\DD(\cT_h)^m: \cI_h (v_h\cdot \tu_h)=0\}$
imply the invertibility of $DF(\tu_h,\tl_h)$, cf.~\cite{Babu70,Brez74}.
The $W^{1,\infty}$ stability of $\cI_h$ and the $L^2$ stability of
$\cJ_h$ imply that the bounds on $DF(\tu_h,\tl_h)$ depend on $\|u\|_{W^{1,\infty}}$
and $\|\l\|$.
\end{proof}

The second auxiliary result bounds the operator norm of $F(\tu_h,\tl_h)$
for interpolants of a regular harmonic map $(u,\l)$.

\begin{lemma}[Residual of interpolants]\label{la:smallness}
Assume that a harmonic map $(u,\l)\in\cA$ satisfies~\eqref{eq:reg_cond}
and define $(\tu_h,\tl_h) \in \cA_h$ via $\tu_h= \cI_hu$ and $\tl_h = \cJ_{h,\DD} \l$.
We then have that
\[
\big|F_h(\tu_h,\tl_h)[(v_h,\mu_h)]\big|
\le ch \|(v_h,\mu_h)\|_{X_h},
\]
where $c$ depends on $\|D^2u\|$ and $\|u\|_{W^{1,\infty}}$.
\end{lemma}

\begin{proof}
The pair $(u,\l)$ satisfies for all $(v,\mu)\in X$ the identity
$F(u,\l)[(v,\mu)]=0$, where
\[
F(u,\l)[(v,\mu)] = (\nabla u,\nabla v) + (\l,u\cdot v) + (\mu, |u|^2 - 1).
\]
Since $X_h\subset X$ we thus have that
\[\begin{split}
\big|F_h(\tu_h,\tl_h) & [(v_h,\mu_h)]\big|
= \big|F_h(\tu_h,\tl_h)[(v_h,\mu_h)]- F(u,\l)[(v_h,\mu_h)] \big| \\
&\le \big|(\nabla [\tu_h-u],\nabla v_h)\big|
+ \big|(\tl_h,\tu_h\cdot v_h)_h - (\l,u\cdot v_h)\big| = I+II,
\end{split}\]
where we used that $|u|^2 = \cI_h |\tu_h|^2 =1$, so that contributions
involving $\mu_h$ vanish.
For the first term we deduce with nodal interpolation estimates that
\[
I \le  ch \|D^2 u\|\|\nabla v_h\|.
\]
To bound the second term we first note that $\tu_h\cdot v_h|_{\G_\DD} = 0$
so that we may replace $\tl_h$ by $\cJ_h \l$.
With Lemma~\ref{la:quad_control}, inverse estimates, and $\tu_h = \cI_h u$, we find that
\[\begin{split}
II  & \le \big|(\cJ_h \l ,\tu_h\cdot v_h)_h - (\cJ_h \l ,u \cdot v_h) \big|
+ \big| (\cJ_h \l ,u \cdot v_h)) - (\l,u\cdot v_h)\big| \\
&\le c h^2 \big(\|\nabla \cJ_h \l \| \|\nabla \cI_h (u \cdot v_h)\| + \|\cJ_h \l\| \|D_h^2 (u\cdot v_h)\|\big)
+ c h \|\nabla \l\| \|u \cdot v_h\|  \\
&\le c h \|\nabla v_h\| \|\l\|_{H^1} \big(\|u\|_{W^{1,\infty}} + \|D^2 u\|\big).
\end{split}\]
The combination of the estimates implies the result.
\end{proof}

To derive an error estimate using the inverse function theorem a local Lipschitz
continuity property for $DF_h$ is required.

\begin{lemma}[Lipschitz estimate]\label{la:lip_bd}
For all $(u_h,\l_h), (\tu_h,\tl_h) \in \cA_h$ we have
\[
\begin{split}
\big| &DF_h(u_h,\l_h)[(v_h,\mu_h),(w_h,\eta_h)] - DF_h(\tu_h,\tl_h)[(v_h,\mu_h),(w_h,\eta_h)]\big| \\
&\le c \g_\inv(h) \|(u_h-\tu_h,\l_h -\tl_h) \|_{X_h} \|(v_h,\mu_h)\|_{X_h} \|(w_h,\eta_h)\|_{X_h}.
\end{split}
\]
\end{lemma}

\begin{proof}
We have
\[\begin{split}
\big| &DF_h(u_h,\l_h)[(v_h,\mu_h),(w_h,\eta_h)] - DF_h(\tu_h,\tl_h)[(v_h,\mu_h),(w_h,\eta_h)]\big| \\
& \le \big|(\l_h-\tl_h,w_h\cdot v_h)_h  \big|
+ \big|(\mu_h, [u_h-\tu_h]\cdot w_h)_h  \big|
+ \big|(\eta_h,[u_h-\tu_h]\cdot v_h)_h\big|.
\end{split}\]
To estimate the terms on the right-hand side we consider the first term
and use Lemma~\ref{la:quad_control} and inverse estimates to deduce that
\[\begin{split}
\big|(\l_h-\tl_h,w_h\cdot v_h)_h  \big|
& \le  \big| (\l_h-\tl_h, w_h\cdot v_h)  \big| + c h \| (\l_h - \tl_h)\| \|\nabla (w_h \cdot v_h)\| \\
& \le c \|\l_h -\tl_h \|_{H^{-1}} \|\nabla (w_h \cdot v_h)\|  \\
& \le c \g_\inv(h) \|\l_h -\tl_h \|_{H^{-1}}  \|\nabla w_h\| \|\nabla v_h\|.
\end{split}\]
The other terms are estimated analogously.
\end{proof}

The quasi-optimal error estimate results from an application of the inverse function
theorem, cf. Theorem~\ref{quant_inv_fn}.

\begin{theorem}[Error estimate]
Let $u\in H^1(\O;\R^m)$ be a harmonic map such that with $\l =-|\nabla u|^2$
the pair $(u,\l) \in \cA$ satisfies~\eqref{eq:reg_cond} and~\eqref{eq:stable_hm}.
Then, for $h$ sufficiently small,
 there exists a unique solution $(u_h,\l_h) \in \cA_h$ for $F_h(u_h,\l_h)=0$
in a neighborhood $B_\veps(u,\l)$ with $\veps = c \g_\inv(h)^{-1}$ that satisfies
\[
\|\nabla (u -u_h)\| + \|\l-\l_h\|_{H^{-1}} \le c_u h.
\]
\end{theorem}

\begin{proof}
(i) We verify the conditions of the inverse
function theorem. Letting $(\tu_h,\tl_h) = (\cI_h u, \cJ_{h,\DD} \l)$
we have the smallness result from Lemma~\ref{la:smallness} with $\k = c_u h$,
the Lipschitz estimate from Lemma~\ref{la:lip_bd} with $c_L = c_u \g_\inv(h)$,
the invertibility result from Lemma~\ref{la:invert} with
$c_\inv = c_u$. Hence, within $B_\veps(\tu_h,\tl_h)$ for every $\veps>0$ with
$c_u h \le \veps \le c_u' \g_\inv(h)^{-1}$ there exists a unique solution
$(u_h,\l_h) \in X_h$ with $F_h(u_h,\l_h) = 0$. \\
(ii) To derive the error estimate we first note that we may choose
$\veps = c_h h$ so that
\[
\|\nabla (u_h-\tu_h)\| + \|\l_h - \tl_h\|_{H^{-1}} \le c h.
\]
We have $\|\nabla (u-\tu_h)\|\le c h$.
To bound the quasi-interpolation error $\|\l - \tl_h\|_{H^{-1}}$ we
define $\d_h = \cJ_{h,\DD}\l - \cJ_h \l$ and
note that
\[
(\d_h ,\phi)
= (\d_h, \phi - \cJ_{h,\DD} \phi)
+ (\d_h, \cJ_{h,\DD} \phi)
- (\d_h , \cJ_{h,\DD} \phi)_h,
\]
where we used that the last term vanishes. With estimates for the
Cl\'ement quasi-interpolant and Lemma~\ref{la:quad_control} we deduce that
\[
|(\d_h ,\phi)| \le c h \|\d_h\| \|\nabla \phi\|
+ ch \|\d_h \| \|\nabla \cJ_{h,\DD} \phi\|.
\]
Inverse estimates and $H^1$ and $L^2$ stability properties
of the Cl\'ement quasi-interpolant $\cJ_{h,\DD}$ thus imply that
\[
\|\d_h\|_{-1} \le c h \|\d_h\| \le c h  \|\l\|.
\]
Noting $\|\l-\cJ_h \l\|\le ch \|\nabla \l\|$ we find that
$\|\l-\cJ_{h,\DD} \l\|_{H^{-1}} \le c h$, which implies the error estimate.
\end{proof}


\section{Other target manifolds}\label{sec:targets}
To discuss the validity of the theory in case of other target manifolds
we consider a hypersurface $\cM\subset \R^m$ given as the zero level set
of a twice continuously differentiable function $g:\R^m \to \R$, i.e.,
\[
\cM = \{s \in \R^m: g(s) =0 \}.
\]
We assume that $Dg$ is nonvanishing on $\cM$; the kernel of $Dg$ defines
the tangent space of $\cM$.
Harmonic maps into $\cM$ are then defined as stationary configurations
$(u,\l)\in \cA$ for the Lagrange functional
\[
L(u,\l) = \frac12 \int_\O |\nabla u|^2 \dv{x} + \int_\O \l \, g(u) \dv{x},
\]
where the last term is interpreted as the application of $\l$ to $g(u)$.
Stationary points $(u,\l)$ satisfy $F(u,\l) = 0$, where
\[
F(u,\l)[(v,\mu)] = (\nabla u,\nabla v) + (\l,Dg(u)\cdot v) + \dual{\mu}{g(u)}.
\]
Crucial for the application of the inverse function theorem are the
invertibility and continuity properties of the second variation
of $I$ given by
\[ \begin{split}
DF(u,\l)[(v,\mu),(w,\eta)]
&= (\nabla v,\nabla w) + (\l, D^2g(u)[v,w])  \\
& \qquad + \dual{\mu}{Dg(u)\cdot w} + \dual{\eta}{Dg(u)\cdot v}.
\end{split}\]
The invertibility
of $DF$ can be analyzed as in the case of the unit sphere using
$v_h = \cI_h ( (\tPi_h \phi) |Dg(u)|^{-2} Dg(u))$ to establish the inf-sup
condition.
A local Lipschitz continuity property requires bounding the difference
\begin{equation}\label{eq:critical_term}
\begin{split}
 \big|\dual{\l}{D^2 g(u)& [v,w]} - \dual{\tl}{D^2 g(\tu)[v,w]}\big| \\
& \le \|\l-\tl\|_{H^{-1}} \| \nabla (D^2 g(u)[v,w])\| \\
& \qquad\ + \|\tl\|_{H^{-1}} \|\nabla((D^2g(u)-D^2g(\tu))[v,w])\|.
\end{split}
\end{equation}
We have, e.g.,
\[\begin{split}
\| \nabla (D^2 g(u)[v,w])\|
& \le \|D^3 g(u)\|_{L^\infty} \|\nabla u\| \|v\|_{L^\infty} \|w\|_{L^\infty} \\
& \qquad + c \| D^2 g(u)\|_{L^\infty} \|\nabla v\| \|\nabla w\|.
\end{split}\]
Bounding the first term on the right-hand side in a discrete setting using
the $H^1_\DD$ norms of $v$ and $w$
requires applying the inverse estimate~\eqref{eq:inv_est_inf} twice,
which leads to $c_L' \le c \g_\inv(h)^2$. If $d=1$ or $d=2$ this still allows
us to apply the inverse function theorem, cf.~\cite{QiTaWi09},
while if $d=3$ it is in general
not guaranteed that $2 c_\inv \k \le 1/(2 c_\inv c_L')$ as both, $\g_\inv(h)^{-2}$ and
$\k$, are of order $O(h)$.
A positive case corresponds to boundaries of ellipsoids for which $g$ can be chosen
as a quadratic function so that $D^2 g$ is constant and the right-hand
side in~\eqref{eq:critical_term} simplifies. Slightly more general,
it suffices to require that $D^3 g$ is sufficiently small and assuming
that we have the additional regularity property $u\in W^{2,\infty}(\O;\R^m)$.

\section{Numerical experiments}\label{sec:num_ex}
In this section we experimentally investigate the validity of the error
estimate and the related aspect of the convergence properties of the Newton
scheme for nonsingular $S^2$-valued harmonic maps in two- and three-dimensional
settings. The first example is obtained from the stereographic projection.

\begin{example}[Inverse stereographic projection]\label{ex:inv_stereo}
Let $d=2$ and $\O = (-1/2,1/2)^2$, $\G_\DD = \p\O$, and $u_\DD = \pi_{\rm st}^{-1}|_{\p\O}$
with the inverse stereographic projection $\pi_{\rm st}^{-1}: \O \to S^2$
given for $x\in \O$ by
\[
\pi_{\rm st}^{-1}(x) = (|x|^2 + 1)^{-1} \begin{bmatrix} 2 x \\ 1 - |x|^2 \end{bmatrix}.
\]
Then $u = \pi_{\rm st}^{-1}$ is a harmonic map with $u|_{\p\O} = u_\DD$.
\end{example}

The second example considers the prototypical harmonic map $x\mapsto x/|x|$,
$x\in \R^3$, away from the origin to avoid a singular solution.

\begin{example}[Radial projection]\label{ex:radial}
Let $d=3$, $\O = (-1/2,1/2)^3$, $\G_\DD = \p\O$, and for $s=0.9 e_3$ and 
and $x\in \p\O$
\[
u_\DD(x) = \frac{x-s}{|x-s|}.
\]
Then $u(x) = (x-s)/|x-s|$ is a harmonic map with $u|_{\p\O} = u_\DD$.
\end{example}

The sufficient condition for global $H^1$ coercivity 
$|\nabla u|< c_P^{-1} \le \pi$ is satisfied in the first and violated
in the second example. Visualizations of
numerical solutions for the examples are displayed in Figure~\ref{fig:snaps};
they illustrate that the cut-locus condition is satisfied in both cases. 
To iteratively compute discrete harmonic maps, we use the Newton scheme
which computes for an initial pair $(u_h^0,\l_h^0)\in \cA_h$ the iterates
$(u_h^k,\l_h^k) \in \cA_h$ via the corrections $(d_h^k,\d_h^k) \in X_h$
that solve
\[
DF_h(u_h^k,\l_h^k)[(v_h,\mu_h),(d_h^k,\d_h^k)] = - F_h(u_h^k,\l_h^k)[v_h,\mu_h]
\]
for all $(v_h,\mu_h) \in X_h$ and the update
\[
(u_h^{k+1},\l_h^{k+1}) = (u_h^k,\l_h^k) + (d_h^k,\d_h^k),
\]
until $\|\nabla d_h^k\| + \|\d_h^k\| \le \veps_{\rm stop}$. We always use
$\veps_{\rm stop} = 10^{-10}$ and denote the final output by $(u_h,\l_h)$.

\begin{figure}[h]
\includegraphics[width=0.48\linewidth]{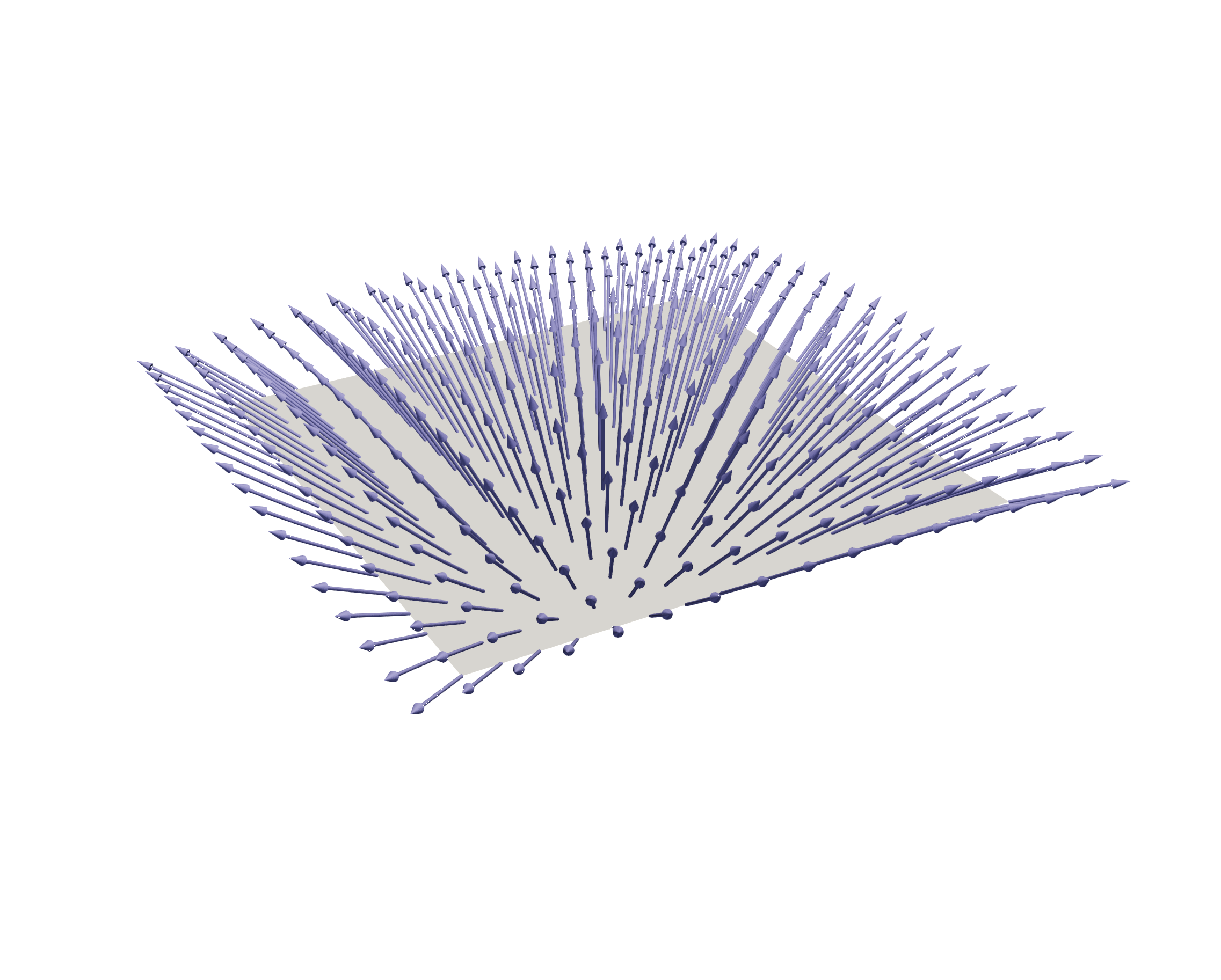} \hspace{2mm}
\includegraphics[width=0.48\linewidth]{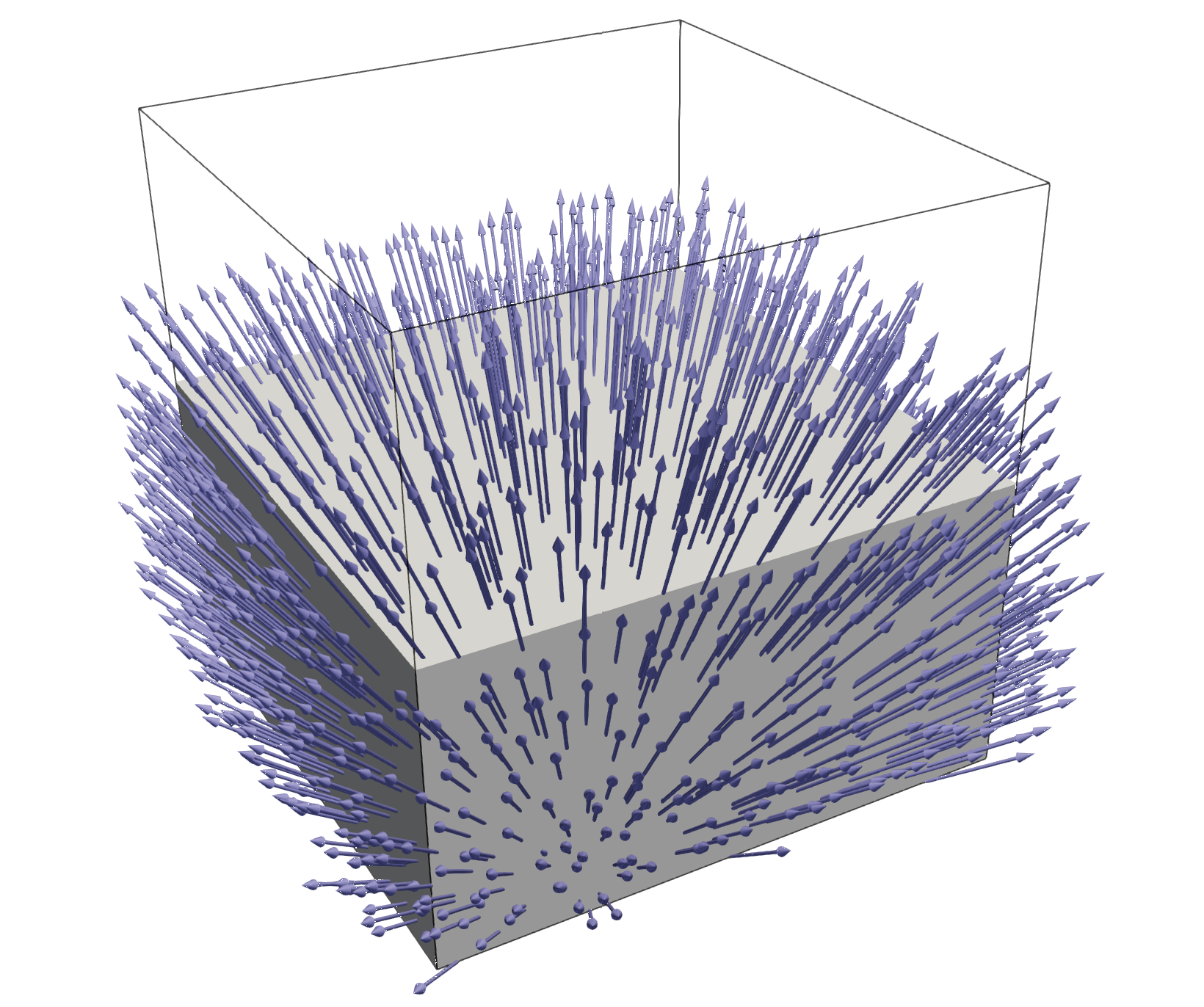}
\caption{\label{fig:snaps} Numerical solutions in
Examples~\ref{ex:inv_stereo} (left) and~\ref{ex:radial} (right).}
\end{figure}

\subsection{Experimental convergence rates}
We use sequences of uniformly refined triangulations of the domains
$\O = (-1/2,1/2)^d$ into triangles or tetrahedra obtained from $\ell$
uniform refinements and with maximal mesh sizes $h_\ell$ comparable
to $2^{-\ell}$.
We refer to these triangulations and quantities related to it via
an index $\ell$ instead of $h_\ell$.
We computed approximate solutions in Examples~\ref{ex:inv_stereo}
and~\ref{ex:radial} and determined the discrete approximation errors
\[
\|e_\ell\|_X = \|\nabla (u_\ell - \cI_\ell u)\| + \|\l_\ell - \cI_{\ell,\DD} \l\|_{H^{-1}_h},
\]
as well as the approximation errors of the Lagrange multiplier in
$L^2$ and $H^{-1}$ norms, i.e.,
\[
\|e^\l_\ell\| = \|\l_\ell - \cI_{\ell,\DD} \l\|, \quad
\|e^\l_\ell\|_{H^{-1}} = \|\l_\ell - \cI_{\ell,\DD} \l\|_{H^{-1}_h}.
\]
Here, $\cI_{\ell,\DD}$ denotes the nodal interpolant with vanishing nodal
values on $\G_\DD$.
We approximated the $H^{-1}$ norm of a finite element function $\mu_h\in \cS^1_\DD(\cT_h)$
by the equivalent quantity $\|\mu_h\|_{H_h^{-1}}= \|\nabla (-\Delta_{h,\DD})^{-1} \mu_h\|$ with the
finite element approximation $(-\Delta_{h,\DD})^{-1}$ of the inverse
of the negative Laplace operator subject to homogeneous Dirichlet boundary
conditions on $\G_\DD$. Experimental convergence rates for an error quantity $\d_\ell$
were determined via the logarithmic slopes given by
\[
\eoc (\d_\ell) = \frac{\log \big(\d_\ell/ \d_{\ell-1}\big)}{\log(h_\ell/h_{\ell-1})}.
\]
For sequences of uniform triangulations in two dimensions obtained from
red refinements of the triangles we have $h_\ell/h_{\ell-1} = 1/2$.
Table~\ref{tab:errs_ex_2da} displays the full approximation errors
for a sequence of uniform triangulations with nodes $\cN_\ell$ and the experimental convergence
rates for different error quantities. We observe a superconvergence phenomenon
in the form of a quadratic rate for the full approximation error. The discrete
Lagrange multipliers converge with respect to the $L^2$ norm with the suboptimal
experimental rate approximately~$0.5$. The same quantities were
computed on a sequence of uniformly refined triangulations with reduced
symmetry properties. These were obtained by randomly perturbing the midpoints of
edges that define the vertices of new triangles. The results shown in
Table~\ref{tab:errs_ex_2db} reveal that this eliminates the superconvergence
phenomenon. Because of the higher complexity of three-dimensional triangulations
and the lack of symmetry properties of the exact solution a larger preasymptotic
range is expected in the three-dimensional setting of Example~\ref{ex:radial}.
The results shown in Table~\ref{tab:errs_ex_3d} indicate a tendency to a linear
convergence behavior on the employed sequence of unperturbed uniform triangulations;
the Lagrange multipliers appear to converge at optimal rates in~$H^{-1}$ as well
as in~$L^2$.

\begin{table}[h]
  \footnotesize
  \begin{tabular}{|c r S[table-format=1.3e-1] S[table-format=1.3e-1] S[table-format=1.3e-1] S[table-format=1.3e-1]|}
  \hline
      {$\ell$} & { $\#\cN_\ell$} & {$\|e_\ell\|_X$}      & {$\eoc(\|e_\ell^\l\|)$} & {$\eoc(\|e_\ell^\l\|_{H_h^{-1}})$} & {$\eoc(\|e_\ell\|_X)$} \\ \hline
  1        & 9              & 0.20000000000000018   & 0.0                   & 0.0                              & 0.0 \\
  2        & 25              & 0.0663180364347164    & 0.4149512824075641    & 1.9164402947595718               & 1.5925268031899233 \\
  3        & 81             & 0.01725969320191879   & 0.464317164571432     & 1.9602765598401453               & 1.9419944714450068 \\
  4        & 289             & 0.004359651236203938  & 0.48401670194531604   & 1.9895934945633842               & 1.9851221883771348 \\
  5        & 1089            & 0.0010927590888991517 & 0.49254715852277636   & 1.997365576988211                & 1.9962373491376362 \\
  6        & 4225           & 0.0002733686974838861 & 0.4964226668140132    & 1.9993392037905315               & 1.9990554181939597 \\
  7        & 16641           & 6.83533748326938e-5   & 0.4982507039084831    & 1.9998347156868022               & 1.999763578979525 \\ \hline
  \end{tabular}
\caption{\label{tab:errs_ex_2da} Approximation errors and experimental convergence rates in Example~\ref{ex:inv_stereo} on a sequence of
uniformly refined triangulations consisting of right-angled triangles. A superconvergence phenomenon is
observed for the full approximation error, suboptimal convergence occurs for the Lagrange multiplier in $L^2$.}
\end{table}

\begin{table}[h]
  \footnotesize
  \begin{tabular}{|c r S[table-format=1.3e-1] S[table-format=1.3e-1] S[table-format=1.3e-1] S[table-format=1.3e-1]|}
  \hline
  {$\ell$} & {$\#\cN_\ell$} & {$\|e_\ell\|_X$}     & {$\eoc(\|e^\l_\ell\|)$} & {$\eoc(\|e^\l_\ell\|_{H_h^{-1}})$} & {$\eoc(\|e_\ell\|_X)$} \\ \hline
  1        & 9             & 0.22424173348758747  & 0.0                   & 0.0                              & 0.0 \\
  2        & 25             & 0.059630300417115575 & 0.4724114878151385    & 3.056820430067094                & 2.2275269132594744 \\
  3        & 81            & 0.02498428172816025  & 0.5244337341959158    & 1.151496104136014                & 1.3919453485176276 \\
  4        & 289            & 0.010683961355861252 & 0.4914399086052459    & 1.3608444424299542               & 1.3185805183868773 \\
  5        & 1089           & 0.005362561677156704 & 0.5333351832218219    & 1.172820181856657                & 1.1757055620431447 \\
  6        & 4225          & 0.002966558981164642 & 0.529762487880129     & 0.9451936977856115               & 0.9128352269680083 \\
  7        & 16641          & 0.001426930327305059 & 0.5178425432086236    & 1.1821764070977479               & 1.1902578943095081 \\ \hline
  \end{tabular}
\caption{\label{tab:errs_ex_2db} Approximation errors and experimental convergence rates in Example~\ref{ex:inv_stereo} on a sequence of
uniformly refined triangulations consisting of perturbed right-angled triangles. No superconvergence phenomenon occurs and the
theoretically predicted rates are confirmed, suboptimal convergence occurs for the Lagrange multiplier in $L^2$.}
\end{table}

\begin{table}[h]
  \footnotesize
  \begin{tabular}{|c r S[table-format=1.3e-1] S[table-format=1.3e-1] S[table-format=1.3e-1] S[table-format=1.3e-1]|}
  \hline
  {$\ell$} & {$\#\cN_\ell$} & {$\|e_\ell\|_X$}    & {$\eoc(\|e^\l_\ell\|)$} & {$\eoc(\|e^\l_\ell\|_{H_h^{-1}})$} & {$\eoc(\|e_\ell\|_X)$} \\ \hline
  1        & 27             & 0.12406689702991758 & 0.0                   & 0.0                              & 0.0 \\
  2        & 125            & 0.14238022363400532 & 0.3615813152281997    & -0.9107731682226866              & -0.2807424383554023 \\
  3        & 729           & 0.14522729032887494 & 0.16898122141718536   & 0.270000724054067                & -0.041926371828678824 \\
  4        & 4913          & 0.10949509501480362 & -0.6774392660729502   & 0.7367785631635176               & 0.8583474371445884 \\
  5        & 35937         & 0.07787835985889435 & -0.7586225397386103   & 0.6732651450276272               & 0.7196571247303444 \\ \hline
  \end{tabular}
\caption{\label{tab:errs_ex_3d} Approximation errors and experimental convergence rates in the three-dimensional setting of
Example~\ref{ex:radial} on a sequence of uniformly refined triangulations.}
\end{table}

\subsection{Iteration convergence}
The conditions of the inverse function theorem imply the superlinear
convergence of Newton type iterations provided that the starting
value is sufficiently close to the solution. In order to experimentally
determine the size of this neighborhood and to quantify the
convergence speed, we use oscillating perturbations of the nodal interpolants
of the exact solutions as starting values, i.e.,
\[
u_h^0 = \cI_h u + \xi_h, \quad \l_h^0 = \cI_{h,\DD} \l + \zeta_h.
\]
The vectorial and scalar perturbations are given by
\[
\xi_h = n_{f,\vrho} (x) [1,\dots,1]^\transp, \quad \zeta_h = n_{f,\vrho}(x),
\]
where for a given frequency $f$ and strength $\vrho \ge 0$ the noise function
$n_{f,\vrho}$ is given by
\[
n_{f,\vrho}(x) =  \vrho \sin(2 \pi f x_1) \dots \sin(2 \pi f x_d).
\]
We experimentally investigated the experimental convergence behavior
of the Newton iteration by representing the residual $F_h(u_h^k,\l_h^k)$
in the nodal basis of the finite element spaces and computing its
Euclidean norm. Table~\ref{tab:newton_2db}
displays the decay of the residuals and indicates a superlinear
but non-quadratic convergence behavior in the two-dimensional settting
of Example~\ref{ex:inv_stereo} with a perturbed triangulation $\cT_7$.  
The perturbation parameters were chosen as $f=10$ and $\vrho= h_\ell$.

\begin{table}[h]
  \footnotesize
  \begin{tabular}{|c S[table-format=1.3e1] S[table-format=-1.3e-1] S[table-format=1.3e-1]| }
  \hline
  {step $k$} & {time (s)}         & {$\mathrm{res}_{k-1}$} & {$\mathrm{res}_{k-2}/\mathrm{res}_{k-1}$} \\ \hline
  0          & 0.429122761        & 1.0                    & 0 \\
  1          & 115.320855653      & 8.128032060123092      & 8.128032060123092 \\
  2          & 1337.492493601     & 0.0014476335447287877  & 0.00017810381824538037 \\
  3          & 1500.231272498     & 1.8973630767054442e-5  & 0.013106653155518807 \\
  4          & 3470.2993826780003 & 5.328912424507504e-7   & 0.028085886617760875 \\
  5          & 5462.661815639     & 9.427732807989806e-10  & 0.0017691663996263014 \\
  6          & 7513.185854847001  & 3.867460537284524e-12  & 0.004102216955074217 \\ \hline
\end{tabular}
\caption{\label{tab:newton_2db} Iterations of the Newton iteration on the
perturbed triangulation $\cT_7$ in Example~\ref{ex:inv_stereo}
with norms of resdiuals $\mathrm{res}_k \simeq F_h(u_h^k,\l_h^k)$.
Their quotients indicate a superlinear, non-quadratic convergence behavior. }
\end{table}

To experimentally determine the convergence area of the Newton iteration as
neighborhoods of the interpolants $x_h = (\cI_h u, \cI_{h,\DD} \l)$ we used
perturbations of $x_h$ of increasing size, i.e.,
\[
f = 10, \quad \vrho = h, \, h^{3/4}, \, h^{1/2}, \, h^{1/4}, \, h^0.
\]
Tables~\ref{tab:newton_area_2d} and ~\ref{tab:newton_area_3d}
display the iteration numbers required to achieve the stopping criterion
on the fixed triangulations $\cT_7$ and $\cT_5$ for Examples~\ref{ex:inv_stereo}
and~\ref{ex:radial}, respectively. A hyphen indicates that the criterion was
not satisfied within~25 iterations.

\begin{table}[h]
  \footnotesize
  \begin{tabular}{|c c c c c c c| }
  \hline
  $\ell$ & $ \vrho = 0$\; & $h$ & $h^{3/4}$ & $h^{1/2}$ & $h^{1/4}$ & $h^0$ \\ \hline
  1                             & 2         & 2   & 3         & 3         & 2         & 3 \\
  2                             & 2         & 4   & 2         & 5         & 5         & 5 \\
  3                             & 2         & 5   & 5         & 8         & 8         & --- \\
  4                             & 2         & 4   & 5         & 6         & ---       & --- \\
  5                             & 2         & 5   & 5         & 6         & ---       & --- \\
  6                             & 2         & 5   & 5         & 6         & ---       & --- \\
  7                             & 3         & 6   & 6         & 6         & ---       & --- \\ \hline
  \end{tabular}
  \caption{\label{tab:newton_area_2d} Iteration numbers for the Newton method in the two-dimensional
    Example~\ref{ex:inv_stereo}
    for different perturbations of strength $\vrho$ of the nodal interpolants as starting value.}
\end{table}

\begin{table}[h]
  \footnotesize
  \begin{tabular}{|c c c c c c c| }
  \hline
  $\ell$ & $\vrho = 0$\; & $h$ & $h^{3/4}$ & $h^{1/2}$ & $h^{1/4}$ & $h^0$ \\ \hline
  1                             & 3         & 3   & 3         & 3         & 3         & 3 \\
  2                             & 3         & 3   & 3         & 3         & 3         & 3 \\
  3                             & 3         & 11  & 11        & ---       & ---       & --- \\
  4                             & 3         & 9   & 12        & ---       & ---       & --- \\
  5                             & 4         & 6   & 8         & ---       & ---       & --- \\ \hline
  \end{tabular}
  \caption{\label{tab:newton_area_3d} Iteration numbers for the Newton method in the three-dimensional
  Example~\ref{ex:radial} for different perturbations of strength $\vrho$ of the nodal interpolants as starting value.}
\end{table}


\medskip

\subsection*{Acknowledgments} The authors thank Tobias Lamm for stimulating
discussions. Financial support by the German Research Foundation (DFG)
via research unit FOR 3013 {\em Vector- and tensor-valued surface PDEs}
(Grant no. BA2268/6–1) is gratefully acknowledged.


\section*{References}
\printbibliography[heading=none]

\end{document}